\documentclass[reqno]{amsart}
\usepackage[backend=bibtex, sorting=none, bibstyle=alphabetic, citestyle=alphabetic, sorting=nyt]{biblatex}  
\bibliography{ref.bib}
\usepackage{amssymb}
\usepackage{calrsfs}
\usepackage{graphics,graphicx}
\usepackage{enumerate}
\usepackage{enumitem}
\usepackage{url}
\usepackage{lipsum}
\usepackage{xcolor}
\usepackage[ruled, lined, linesnumbered, longend]{algorithm2e}
\usepackage{hyperref}

\newtheorem{theorem}{Theorem}[section]
\newtheorem{lemma}[theorem]{Lemma}
\newtheorem{proposition}[theorem]{Proposition}

\newtheorem{corollary}[theorem]{Corollary}
\numberwithin{equation}{section}

\theoremstyle{remark}
\newtheorem*{remark}{Remark}

\DeclareMathOperator{\li}{li}

\def\reals{\hbox{\rm I\kern-.18em R}}
\def\complexes{\hbox{\rm C\kern-.43em
\vrule depth 0ex height 1.4ex width .05em\kern.41em}}
\def\field{\hbox{\rm I\kern-.18em F}} 

\newcommand\blfootnote[1]{%
  \begingroup
  \renewcommand\thefootnote{}\footnote{#1}%
  \addtocounter{footnote}{-1}%
  \endgroup
}

\allowdisplaybreaks

\begin{document}

\title[Improving bounds on prime counting functions]{Improving bounds on prime counting functions by partial verification of the Riemann hypothesis}

\author{Daniel R. Johnston}

\begin{abstract}
    Using a recent verification of the Riemann hypothesis up to height $3\cdot 10^{12}$, we provide strong estimates on $\pi(x)$ and other prime counting functions for finite ranges of $x$. In particular, we get that $|\pi(x)-\li(x)|<\sqrt{x}\log x/8\pi$ for $2657\leq x\leq 1.101\cdot 10^{26}$. We also provide weaker bounds that hold for a wider range of $x$, and an application to an inequality of Ramanujan.
\end{abstract}

\maketitle

\blfootnote{\textit{Affiliation}: School of Science, The University of New South Wales, Canberra, Australia}
\blfootnote{\textit{Email}: daniel.johnston@adfa.edu.au}

\section{Introduction}
In 1976, Schoenfeld \cite[Corollary 1]{schoenfeld1976sharper} proved that under assumption of the Riemann hypothesis,
\begin{equation}\label{schoeq}
    |\pi(x)-\li(x)|<\frac{\sqrt{x}}{8\pi}\log x,
\end{equation}
for $x\geq 2657$. Although a complete proof of the Riemann hypothesis remains out of reach, partial results can be used to prove \eqref{schoeq} for a finite range. In this direction, we prove that \eqref{schoeq} holds provided $x\geq 2657$ and
\begin{equation}\label{djbound}
    \frac{9.06}{\log\log x}\sqrt{\frac{x}{\log x}}\leq T.
\end{equation}
Here, $T$ is the largest known value such that the Riemann hypothesis is true for all zeros $\rho$ with $\Im(\rho)\in(0,T]$. A recent computation by Platt and Trudgian \cite{platt2021riemann} allows us to take $T=3\cdot 10^{12}$. Substituting this $T$ into \eqref{djbound} tells us that $\eqref{schoeq}$ holds for all $2657\leq x\leq 1.101\cdot 10^{26}$.

These results improve on earlier work by B{\"u}the \cite[Theorem 2]{buthe2016estimating}, who proved that \eqref{schoeq} holds provided $x\geq 2657$ and
\begin{equation}\label{buthebound}
    4.92\sqrt{\frac{x}{\log x}}\leq T.
\end{equation}
In particular, provided $T\geq 46$, \eqref{djbound} holds for a wider range of $x$ than $\eqref{buthebound}$. So by comparison, one only obtains $x\leq 2.169\cdot 10^{25}$ using \eqref{buthebound} with $T=3\cdot 10^{12}$.

To prove \eqref{djbound} we use B{\"u}the's original method with an additional iterative argument and several other optimisations. Similar to B{\"u}the, we prove corresponding bounds for the prime counting functions $\theta(x)$, $\psi(x)$ and $\Pi(x)$.

In Section \ref{sectnot} we list all the definitions and lemmas that we will use from \cite{buthe2016estimating}. In Section \ref{sectmain} we prove the main result. Then, in Section \ref{sectimprov} we discuss possible improvements and variations. For instance, we show (Theorem \ref{weakthm}) that the weaker bound $|\pi(x)-\li(x)|<\sqrt{x}\log x$ holds for $2\leq x\leq 2.165\cdot 10^{30}$. Finally, in Section \ref{appsec} we discuss an application to an inequality of Ramanujan.

\section{Notation and setup}\label{sectnot}
Throughout this paper, we work with the normalised prime counting functions
\begin{align}\label{chebfuncs}
    &\pi(x)=\sideset{}{^*}\sum_{p\leq x} 1, \qquad\qquad\ \ \Pi(x)=\sideset{}{^*}\sum_{p^m\leq x}\frac{1}{m},\notag\\
    &\theta(x)=\sideset{}{^*}\sum_{p\leq x}\log p, \qquad\psi(x)=\sideset{}{^*}\sum_{p^m\leq x}\log p,
\end{align}
where $\sum^*$ indicates that the last term in the sum is multiplied by $1/2$ when $x$ is an integer. However, we note that our main results (Theorem \ref{mainthm} and Theorem \ref{weakthm}), will also hold for the standard (unnormalised) prime counting functions. Our main focus will be on the function $\psi(x)$ since the other functions in \eqref{chebfuncs} can be related to $\psi(x)$ by simple bounding and partial summation arguments.

Following \cite[Section 2]{buthe2016estimating}, we define
\begin{align*}
    \ell_{c,\varepsilon}(\xi)=\frac{c}{\sinh(c)}\frac{\sin(\sqrt{(\xi\varepsilon)^2-c^2})}{\sqrt{(\xi\varepsilon)^2-c^2}},\quad a_{c,\varepsilon}(\rho)=\frac{1}{\ell_{c,\varepsilon}(i/2)}\ell_{c,\varepsilon}\left(\frac{\rho}{i}-\frac{1}{2i}\right)
\end{align*}
for $c,\varepsilon>0$. We will also make use of the auxiliary function $\psi_{c,\varepsilon}(x)$ defined on page 2484 of \cite{buthe2016estimating}. Notably, $\psi_{c,\varepsilon}(x)$ is a continuous approximation to $\psi(x)$. Moreover, for $x\geq 10$ and $0<\varepsilon\leq 10^{-4}$, we have \cite[Proposition 2]{buthe2018analytic}
\begin{equation}\label{vonmaneq}
    x-\psi_{c,\varepsilon}(x)=\sum_{\rho}\frac{a_{c,\varepsilon}(\rho)}{\rho}x^\rho+\Theta(2).
\end{equation}
Here, $\Theta(2)$ indicates a constant with absolute value less than 2, and the sum is taken over all the non-trivial zeros of the Riemann zeta-function and computed as
\begin{equation*}
    \lim_{T\to\infty}\sum_{|\Im(\rho)|<T}\frac{a_{c,\varepsilon}(\rho)}{\rho}x^\rho.
\end{equation*}
To obtain an expression for $|\psi(x)-x|$ we thus need bounds on $\sum_{\rho}\frac{a_{c,\varepsilon}(\rho)}{\rho}x^\rho$ and $|\psi(x)-\psi_{c,\varepsilon}(x)|$. We will make use of the following collection of lemmas taken from \cite{buthe2016estimating} with slight modifications.
\begin{lemma}[{\cite[Proposition 3]{buthe2016estimating}}]\label{butprop3}
    Let $x>1$, $\varepsilon\leq 10^{-3}$ and $c\geq 3$. Then
    \begin{equation*}
        \sum_{|\Im(\rho)|>\frac{c}{\varepsilon}}\left|a_{c,\varepsilon}(\rho)\frac{x^\rho}{\rho}\right|\leq 0.16\frac{x+1}{\sinh(c)}e^{0.71\sqrt{c\varepsilon}}\log(3c)\log\left(\frac{c}{\varepsilon}\right).
    \end{equation*}
    Moreover, if $a\in (0,1)$ with $a\frac{c}{e}\geq 10^3$, and the Riemann hypothesis holds for all zeros $\rho$ with $\Im(\rho)\in(0,\frac{c}{\varepsilon}]$, then
    \begin{equation*}
        \sum_{a\frac{c}{\varepsilon}<|\Im(\rho)|\leq\frac{c}{\varepsilon}}\left|a_{c,\varepsilon}(\rho)\frac{x^\rho}{\rho}\right|\leq\frac{1+11c\varepsilon}{\pi ca^2}\log\left(\frac{c}{\varepsilon}\right)\frac{\cosh(c\sqrt{1-a^2})}{\sinh(c)}\sqrt{x}.
    \end{equation*}
\end{lemma}
\begin{lemma}[{\cite[Lemma 3]{buthe2016estimating}}]\label{butlemma3}
    If $t_2\geq 5000$ then
    \begin{equation*}
        \sum_{0<|\Im(\rho)|\leq t_2}\frac{1}{|\Im(\rho)|}\leq\frac{1}{2\pi}\log^2\left(\frac{t_2}{2\pi}\right).
    \end{equation*}
\end{lemma}
\begin{lemma}[{\cite[Proposition 4]{buthe2016estimating}}]\label{butprop4}
    Let $x>100$, $\varepsilon<10^{-2}$ and 
    \begin{equation*}
        B_0=\frac{I_1(c)}{2\sinh(c)}\varepsilon xe^{-\varepsilon}>1,
    \end{equation*}
    where 
    \begin{equation*}
        I_1(x)=\sum_{n=0}^\infty\frac{(x/2)^{2n+1}}{n!\Gamma(n+2)}
    \end{equation*}
    is a modified Bessel function of the first kind. We then have
    \begin{equation*}
        |\psi(x)-\psi_{c,\varepsilon}(x)|\leq e^{2\varepsilon}\log(e^{\varepsilon}x)\left[\frac{\varepsilon x}{\log B_0}\frac{I_1(c)}{\sinh(c)}+2.01\varepsilon\sqrt{x}+\frac{1}{2}\log\log(2x^2)\right].
    \end{equation*}
\end{lemma}
\begin{lemma}[{\cite[Proposition 5]{buthe2016estimating}}]\label{butprop5}
    For $c_0>0$ let
    \begin{equation*}
        D(c_0)=\sqrt{\frac{\pi c_0}{2}}\frac{I_1(c_0)}{\sinh(c_0)}.
    \end{equation*}
    Then
    \begin{equation*}
        \frac{D(c_0)}{\sqrt{2\pi c}}\leq \frac{I_1(c)}{2\sinh(c)}\leq\frac{1}{\sqrt{2\pi c}}
    \end{equation*}
    holds for all $c\geq c_0$.
\end{lemma}
In particular, note that for Lemma \ref{butprop4} we have taken the case $\alpha=0$ in \cite[Proposition 4]{buthe2016estimating}. Moreover, we remark that Brent, Platt and Trudgian \cite[Lemma 8]{brent2020mean} recently showed that Lemma \ref{butlemma3} holds more generally for $t_2\geq 4\pi e$.

\section{Proof of the main result}\label{sectmain}
We begin by stating the bounds obtained using B{\"u}the's result \cite[Theorem 2]{buthe2016estimating} and Platt and Trudgian's computation \cite{platt2021riemann}.
\begin{proposition}\label{buthethm}
    The following estimates hold:
    \begin{align}
        |\psi(x)-x|&<\frac{\sqrt{x}}{8\pi}\log^2(x),&\text{for }59\leq x\leq 2.169\cdot10^{25},\notag\\
        |\theta(x)-x|&<\frac{\sqrt{x}}{8\pi}\log^2(x),&\text{for }599\leq x\leq 2.169\cdot10^{25},\notag\\
        |\psi(x)-x|&<\frac{\sqrt{x}}{8\pi}\log x(\log x-3),&\text{for }5000\leq x\leq 2.169\cdot10^{25},\label{psiminus3}\\
        |\theta(x)-x|&<\frac{\sqrt{x}}{8\pi}\log x(\log x-2),&\text{for }5000\leq x\leq 2.169\cdot10^{25},\label{thetaminus2}\\
        |\Pi(x)-\li(x)|&<\frac{\sqrt{x}}{8\pi}\log x,&\text{for }59\leq x\leq 2.169\cdot10^{25},\notag\\
        |\pi(x)-\li(x)|&<\frac{\sqrt{x}}{8\pi}\log x,&\text{for }2657\leq x\leq 2.169\cdot10^{25}.\notag
    \end{align}
\end{proposition}
\begin{proof}
    The $2.169\cdot 10^{25}$ comes from substituting $T=3\cdot 10^{12}$ \cite{platt2021riemann} into \cite[Theorem 2]{buthe2016estimating}. Note that \eqref{psiminus3} and \eqref{thetaminus2} do not appear in the statement of B{\"u}the's theorem but are established as intermediary steps in the proof.
\end{proof}

We now prove the main result of this paper.
\begin{theorem}\label{mainthm}
    Let $T>0$ be such that the Riemann hypothesis holds for zeros $\rho$ with $0<\Im(\rho)\leq T$. Then, under the condition $\frac{9.06}{\log\log x}\sqrt{\frac{x}{\log x}}\leq T$, the following estimates hold:
    \begin{align}
        |\psi(x)-x|&<\frac{\sqrt{x}}{8\pi}\log^2x,&\text{for }x\geq 59,\label{psiin}\\
        |\theta(x)-x|&<\frac{\sqrt{x}}{8\pi}\log^2x,&\text{for }x\geq 599,\label{thetain}\\
        |\Pi(x)-\li(x)|&<\frac{\sqrt{x}}{8\pi}\log x,&\text{for }x\geq 59\label{piin1},\\
        |\pi(x)-\li(x)|&<\frac{\sqrt{x}}{8\pi}\log x,&\text{for }x\geq 2657\label{piin2}.
    \end{align}
\end{theorem}
\begin{proof}
    Throughout this proof we will label specific constants $A$, $B$, $C$, $D$ and $E$. This is done to make it clear where optimisations are being made and to allow us to perform an iterative argument.
    
    Now, by Proposition \ref{buthethm} it suffices to consider $x>A$ where $A=2.169\cdot 10^{25}$. We also initially restrict ourselves to $x$ such that
    \begin{equation}\label{beq}
        \frac{B}{\log\log x}\sqrt{\frac{x}{\log x}}\leq T,
    \end{equation}
    where $B=9.65$ and later reduce the value of $B$. We first prove the bound 
    \begin{equation}\label{strongpsieq}
        |\psi(x)-x|<\frac{\sqrt{x}}{8\pi}\log x(\log x-C),\quad\text{for }x>A,
    \end{equation}
    where $C=2.44$. Next, we define
    \begin{equation}\label{ceeq}
        c(x)=\frac{1}{2}\log x+D,\qquad \varepsilon(x)=\frac{\log^{3/2}x\log\log x}{E\sqrt{x}},
    \end{equation}
    where $D=6$ and $E=16$. To simplify notation we write $c=c(x)$ and $\varepsilon=\varepsilon(x)$. Note that for these choices of $D$ and $E$, we have $c>35$, $\varepsilon<4.9\cdot 10^{-11}$ and
    \begin{equation*}
        \frac{c}{\varepsilon}\leq\left(\frac{1}{2}E+\frac{DE}{\log A}\right)\frac{1}{\log\log x}\sqrt{\frac{x}{\log x}}\leq\frac{B}{\log\log x}\sqrt{\frac{x}{\log x}}\leq T.
    \end{equation*}
    Hence we may assume $\Re(\rho)=\frac{1}{2}$ for zeros $\rho$ with $|\Im(\rho)|\leq\frac{c}{\varepsilon}$.
    
    Now, recall from \eqref{vonmaneq} that
    \begin{equation*}
        x-\psi_{c,\varepsilon}(x)=\sum_{\rho}\frac{a_{c,\varepsilon}(\rho)}{\rho}x^\rho+\Theta(2).
    \end{equation*}
    To prove \eqref{strongpsieq} we split $\sum_{\rho}\frac{a_{c,\varepsilon}(\rho)}{\rho}x^\rho$ into three parts and then bound $|\psi(x)-\psi_{c,\varepsilon}(x)|$. For $|\Im(\rho)|>c/\varepsilon$, Lemma \ref{butprop3} gives
    \begin{align}
        \sum_{|\Im(\rho)|>\frac{c}{\varepsilon}}\left|a_{c,\varepsilon}(\rho)\frac{x^{\rho}}{\rho}\right|&\leq 0.16\frac{x+1}{\sinh(c)}e^{0.71\sqrt{c\varepsilon}}\log(3c)\log\left(\frac{c}{\varepsilon}\right)\notag\\
        &\leq\mathcal{E}_1(x),\label{eps1eq}
    \end{align}
    where 
    \begin{equation*}
        \mathcal{E}_1(x)=0.000032\sqrt{x}\log x\log\log x.
    \end{equation*}
    The inequality in \eqref{eps1eq} follows by noticing that for $x>A$
    \begin{equation*}
        \left(\frac{x+1}{\sinh(c)}\right)/\sqrt{x},\ e^{0.71\sqrt{c\varepsilon}} \text{ and } \log(3c)/\log\log x
    \end{equation*}
    are all decreasing functions and $\log(\frac{c}{\varepsilon})/\log x\leq\frac{1}{2}$. Substituting $x=A$ into 
    \begin{equation*}
        \frac{0.16\frac{x+1}{\sinh(c)}e^{0.71\sqrt{c\varepsilon}}\log(3c)}{\sqrt{x}\log x}\times\frac{1}{2}
    \end{equation*}
    then gives $0.0000316\ldots\leq 0.000032$.
    
    When $\frac{\sqrt{2c}}{\varepsilon}<\Im(\rho)<\frac{c}{\varepsilon}$ we use the second part of Lemma \ref{butprop3} with $a=\sqrt{\frac{2}{c}}$ to obtain
    \begin{align}
        \sum_{\frac{\sqrt{2c}}{\varepsilon}<|\Im(\rho)|\leq\frac{c}{\varepsilon}}\left|a_{c,\varepsilon}(\rho)\frac{x^{\rho}}{\rho}\right|&\leq\frac{1+11c\varepsilon}{2\pi}\log\left(\frac{c}{\varepsilon}\right)\frac{\cosh(c\sqrt{1-a^2})}{\sinh(c)}\sqrt{x}\label{eps2eq1}\\
        &\leq\mathcal{E}_2(x),\label{eps2eq2}
    \end{align}
    where 
    \begin{equation*}
        \mathcal{E}_2(x)=0.0293\sqrt{x}\log x.
    \end{equation*}
    For the inequality in \eqref{eps2eq2} we note that similarly to before $\frac{1+11c\varepsilon}{2\pi}$ is decreasing and $\log\left(\frac{c}{\varepsilon}\right)/\log x\leq\frac{1}{2}$ for $x>A$. Then,
    \begin{align*}
        \frac{\cosh(c\sqrt{1-a^2})}{\sinh(c)}&=\frac{e^{\sqrt{\frac{1}{2}\log x+D}\sqrt{\frac{1}{2}\log x+D-2}}+e^{-\sqrt{\frac{1}{2}\log x+D}\sqrt{\frac{1}{2}\log x+D-2}}}{e^{\frac{1}{2}\log x+D}-e^{-\frac{1}{2}\log x-D}}\\
        &\leq\frac{e^{\frac{1}{2}\log x+D-1}+e^{-\sqrt{\frac{1}{2}\log x+D}\sqrt{\frac{1}{2}\log x+D-2}}}{e^{\frac{1}{2}\log x+D}}\\
        &=\frac{1}{e}+\frac{1}{e^{\sqrt{\frac{1}{2}\log x+D}\sqrt{\frac{1}{2}\log x+D-2}+\frac{1}{2}\log x+D}}
    \end{align*}
    which is also decreasing. Substituting $x=A$ into 
    \begin{equation*}
        \frac{1}{2}\cdot\frac{1+11c\varepsilon}{2\pi}\cdot\left(\frac{1}{e}+\frac{1}{\exp\left(\sqrt{\frac{1}{2}\log x+D}\sqrt{\frac{1}{2}\log x+D-2}+\frac{1}{2}\log x+D\right)}\right)
    \end{equation*}
    then gives $0.0292\ldots\leq 0.0293$.
    
    Next, we consider the range $0<|\Im(\rho)|\leq\frac{\sqrt{2c}}{\varepsilon}$. Note that
    \begin{align}
        a_{c,\varepsilon}(\rho)&=\frac{\sqrt{\varepsilon^2/4+c^2}}{\sinh(\sqrt{\varepsilon^2/4+c^2})}\times\frac{\sin(\sqrt{\Im(\rho)^2\varepsilon^2-c^2})}{\sqrt{\Im(\rho)^2\varepsilon^2-c^2}}\notag\\
        &=\frac{\sqrt{\varepsilon^2/4+c^2}}{\sinh(\sqrt{\varepsilon^2/4+c^2})}\times\frac{\sinh(\sqrt{c^2-\Im(\rho)^2\varepsilon^2})}{\sqrt{c^2-\Im(\rho)^2\varepsilon^2}}\label{imeq}\\
        &\leq\frac{c}{\sinh(c)}\times\frac{\sinh(c)}{c}=1\label{decreasingeq}.
    \end{align}
    In particular, \eqref{imeq} follows since $|\Im(\rho)|\leq\frac{c}{\varepsilon}\leq T$ and \eqref{decreasingeq} follows since $\frac{x}{\sinh(x)}$ is decreasing for $x>0$. Hence $|a_{c,\varepsilon}(\rho)/\rho|\leq 1/|\Im(\rho)|$ and so Lemma \ref{butlemma3} gives
    \begin{align}
        \sum_{0<|\Im(\rho)|\leq\frac{\sqrt{2c}}{\varepsilon}}\left|a_{c,\varepsilon}(\rho)\frac{x^{\rho}}{\rho}\right|&\leq\frac{\sqrt{x}}{2\pi}\log\left(\frac{\sqrt{2c}}{2\pi\varepsilon}\right)^2\notag\\
        &=\frac{\sqrt{x}}{2\pi}\log\left(\frac{E\sqrt{x}\sqrt{\log x+2D}}{2\pi\log^{3/2} x\log\log x}\right)^2\notag\\
        &\leq\frac{\sqrt{x}}{2\pi}\log\left(\frac{E\sqrt{x}\sqrt{\log x+2D\frac{\log x}{\log(A)}}}{2\pi\log^{3/2}x\log\log x}\right)^2\notag\\
        &\leq\frac{\sqrt{x}}{2\pi}\left(\frac{1}{2}\log x+\log(2.8)-\log\log x-\log\log\log x\right)^2\notag\\
        &\leq\frac{\sqrt{x}}{8\pi}\log^2x+\mathcal{E}_3(x),\label{e3eq}
    \end{align}
    where 
    \begin{equation*}
        \mathcal{E}_3(x):=\frac{\sqrt{x}}{2\pi}\left(\frac{1}{2}\log x+\log(2.8)-\log\log x-\log\log\log x\right)^2-\frac{\sqrt{x}}{8\pi}\log^2x.
    \end{equation*}
    We now bound $|\psi(x)-\psi_{c,\varepsilon}(x)|$. By Lemma \ref{butprop5}
    \begin{equation}\label{v0eq}
        \frac{0.98}{\sqrt{2\pi c}}\leq \frac{I_1(c)}{2\sinh(c)}\leq\frac{1}{\sqrt{2\pi c}}.
    \end{equation}
    Combining \eqref{v0eq} and Lemma \ref{butprop4} with our definition \eqref{ceeq} of $\varepsilon$ then gives
    \begin{align}\label{explicitpsibound}
        |\psi(x)-\psi_{c,\varepsilon}(x)|\leq\frac{2.0001\sqrt{x}\log^{5/2}x\log\log x}{E\sqrt{\pi(\log x+2D)}}\log\left(\frac{0.97\sqrt{x}\log^{3/2}x\log\log x}{E\sqrt{\pi(\log x+2D)}}\right)^{-1}\notag\\
        +\frac{2.02}{E}\log^{5/2}x\log\log x+0.51\log x\log\log(2x^2).
    \end{align}
    Since $x>A=2.169\cdot 10^{25}$, we have
    \begin{align*}
        \log\left(\frac{0.97\sqrt{x}\log^{3/2} x\log\log x}{E\sqrt{\pi(\log x+2D)}}\right)\geq\log(\sqrt{x})=\frac{1}{2}\log x.
    \end{align*}
    Hence, dividing the first summand in \eqref{explicitpsibound} by $\sqrt{x}\frac{\log^{3/2}x\log\log x}{\sqrt{\log x+2D}}$ gives
    \begin{align*}
        \frac{2.0001\log x}{E\sqrt{\pi}}\log\left(\frac{0.97\sqrt{x}\log^{3/2}x\log\log x}{E\sqrt{\pi(\log x+2D)}}\right)^{-1}&\leq\frac{2.0001\log x}{E\sqrt{\pi}}\times\frac{2}{\log x}\\
        &=0.141\ldots\leq 0.142.
    \end{align*}
    So if we define
    \begin{equation*}
        \mathcal{E}_4(x)=0.142\sqrt{x}\frac{\log^{3/2}x\log\log x}{\sqrt{\log x+2D}}
    \end{equation*}
    and
    \begin{equation*}
        \mathcal{E}_5(x):=\frac{2.02}{E}\log^{5/2}x\log\log x+0.51\log x\log\log(2x^2)+2
    \end{equation*}
    then
    \begin{equation}\label{e4e5eq}
        |\psi(x)-\psi_{c,\varepsilon}(x)|\leq \mathcal{E}_4(x)+\mathcal{E}_5(x).
    \end{equation}
    Thus, by \eqref{vonmaneq}, \eqref{eps1eq}, \eqref{eps2eq2}, \eqref{e3eq} and \eqref{e4e5eq}
    \begin{equation*}
        |\psi(x)-x|\leq\frac{\sqrt{x}}{8\pi}\log^2x+\mathcal{E}_1(x)+\mathcal{E}_2(x)+\mathcal{E}_3(x)+\mathcal{E}_4(x)+\mathcal{E}_5(x).
    \end{equation*}
    Now consider the function
    \begin{equation*}
        \mathcal{E}(x)=\frac{1}{\sqrt{x}\log x}(\mathcal{E}_1(x)+\mathcal{E}_2(x)+\mathcal{E}_3(x)+\mathcal{E}_4(x)+\mathcal{E}_5(x)).
    \end{equation*}
    Differentiating $\mathcal{E}(x)$ with respect to $y=\log x$ we see that $\mathcal{E}(x)$ is decreasing for $x>A$. Moreover, $\mathcal{E}(A)=-0.0976\ldots<-\frac{C}{8\pi}=-0.0970\ldots$. This proves \eqref{strongpsieq}. Letting $T=3\cdot 10^{12}$ in \eqref{beq} and using Proposition \ref{buthethm} then gives
    \begin{equation}\label{psifirstit}
        |\psi(x)-x|<\frac{\sqrt{x}}{8\pi}\log x(\log x-C),\quad\text{for }5000\leq x\leq 9.68\cdot 10^{25}.
    \end{equation}
    From \eqref{psifirstit}, we also obtain 
    \begin{equation}\label{thetafirstit}
        |\theta(x)-x|<\frac{\sqrt{x}}{8\pi}\log x(\log x-2),\quad\text{for }5000\leq x\leq 9.68\cdot 10^{25}.
    \end{equation}
    To see this, we use recent estimates by Broadbent et al. \cite[Corollary 5.1]{broadbent2021sharper} for $\psi(x)-\theta(x)$. Namely\footnote{This estimate is stated in \cite{broadbent2021sharper} for the unnormalised $\psi$ and $\theta$ functions. However, it also holds for the normalised functions whereby the difference $\psi(x)-\theta(x)$ is at most that in the unnormalised setting.}
    \begin{equation*}
        \psi(x)-\theta(x)<a_1x^{1/2}+a_2x^{1/3},
    \end{equation*}
    where for $x\geq e^{50}\approx 5.18\cdot 10^{21}$, we can take $a_1=1+1.93378\cdot 10^{-8}$ and $a_2=1.01718$. In particular, for $x>A$ we have $\psi(x)-\theta(x)\leq(C-2)\frac{\sqrt{x}}{8\pi}\log x$. Hence \eqref{thetafirstit} holds for $A<x\leq 9.68\cdot 10^{25}$ since
    \begin{equation*}
        |\theta(x)-x|\leq\psi(x)-\theta(x)+|\psi(x)-x|.
    \end{equation*}
    For the remaining values of $x$, we use Proposition \ref{buthethm}.
    
    We now repeat the entire proof with
    \begin{equation*}
        (A,B,C,D,E)=(9.68\cdot 10^{25},9.34,2.43,5,16).
    \end{equation*}
    The error terms then update to (with more precision added this time):
    \begin{align*}
        \mathcal{E}_1(x)&=0.0000839\sqrt{x}\log x\log\log x,\\
        \mathcal{E}_2(x)&=0.02928\sqrt{x}\log x,\\
        \mathcal{E}_3(x)&=\frac{\sqrt{x}}{2\pi}\left(\frac{1}{2}\log x+\log(2.751)-\log\log x-\log\log\log x\right)^2-\frac{\sqrt{x}}{8\pi}\log^2x,\\
        \mathcal{E}_4(x)&=0.1411\sqrt{x}\frac{\log^{3/2}x\log\log x}{\sqrt{\log x+10}},\\
        \mathcal{E}_5(x)&=0.12625\log^{5/2}x\log\log x+0.51\log x\log\log(2x^2)+2,\\
        \mathcal{E}(A)&=-0.0967\ldots
    \end{align*}
    and we get
    \begin{align*}
        |\psi(x)-x|<\frac{\sqrt{x}}{8\pi}\log x(\log x-C),\quad\text{for }5000\leq x\leq 1.03\cdot 10^{26},\\
        |\theta(x)-x|<\frac{\sqrt{x}}{8\pi}\log x(\log x-2),\quad\text{for }5000\leq x\leq 1.03\cdot 10^{26}.
    \end{align*}
    Iterating again with
    \begin{equation*}
        (A,B,C,D,E)=(1.03\cdot 10^{26},9.08,2.42,2.4,16.8)
    \end{equation*}
    followed by
    \begin{equation*}
        (A,B,C,D,E)=(1.096\cdot 10^{26},9.06,2.42,2.34,16.8)
    \end{equation*}
    we get
    \begin{align}
        |\psi(x)-x|<\frac{\sqrt{x}}{8\pi}\log x(\log x-C),\label{finalpsi}\\
        |\theta(x)-x|<\frac{\sqrt{x}}{8\pi}\log x(\log x-2)\label{finaltheta},
    \end{align}
    for $x\geq 5000$ and $\frac{9.06}{\log\log x}\sqrt{\frac{x}{\log x}}\leq T$. Combining \eqref{finalpsi} and \eqref{finaltheta} with Proposition \ref{buthethm} proves \eqref{psiin} and \eqref{thetain}.  Certainly, one could perform further iterations but this would produce a minimal improvement. 
    
    Now, using integration by parts 
    \begin{equation*}
        \li(x)-\li(a)=\frac{x}{\log x}+\int_a^x\frac{\mathrm{d}t}{\log^2t}-\frac{a}{\log a}
    \end{equation*}
    so that by partial summation
    \begin{equation*}
        \pi(x)-\pi(a)=\li(x)-\li(a)-\frac{x-\theta(x)}{\log x}+\frac{a-\theta(a)}{\log a}-\int_a^x\frac{t-\theta(t)}{t\log^2t}\mathrm{d}t.
    \end{equation*}
    Hence, for $5000\leq x$ and $\frac{9.06}{\log\log x}\sqrt{\frac{x}{\log x}}\leq T$,
    \begin{align*}
        |\pi(x)-\li(x)|&\leq\frac{\sqrt{x}}{8\pi}(\log x-2)+\left|\pi(5000)-\li(5000)-\frac{\theta(5000)-5000}{\log(5000)}\right|+\frac{\sqrt{x}}{4\pi}-\frac{\sqrt{5000}}{4\pi}\\
        &=\frac{\sqrt{x}}{8\pi}\log x+4.91...-\frac{\sqrt{5000}}{4\pi}\\
        &=\frac{\sqrt{x}}{8\pi}\log x+4.91...-5.62...\\
        &<\frac{\sqrt{x}}{8\pi}\log x.
    \end{align*}
    Making use of \eqref{finalpsi} as opposed to \eqref{finaltheta} then gives $|\Pi(x)-\li(x)|<\frac{\sqrt{x}}{8\pi}\log x$.
    Combined with Proposition \ref{buthethm} we obtain \eqref{piin1} and \eqref{piin2} thereby completing the proof of the theorem.
\end{proof}
Setting $T=3\cdot 10^{12}$ we obtain the following result.
\begin{corollary}\label{maincor}
    The bounds \eqref{psiin}-\eqref{piin2} hold for $x\leq 1.101\cdot 10^{26}$.
\end{corollary}

\section{Possible improvements and variations}\label{sectimprov}

\subsection{Improvements for larger $T$}
The constant 9.06 appearing in \eqref{djbound} can be lowered if the Riemann hypothesis were verified to a higher height. This is because a higher value of $T$ means that the bounds \eqref{psiin}--\eqref{piin2} hold for larger values of $x$ thereby giving sharper error terms in the proof of Theorem \ref{mainthm}. Table \ref{ttable} lists improvements that one would get by increasing $T$ to $10^{13}$, $10^{14}$ and $10^{15}$. The values in the table were computed using the same iterative method as in the proof of Theorem \ref{mainthm}. 

\def\arraystretch{1.5}
\begin{table}[h]
\centering
\caption{Value of $K$ such that \eqref{psiin}--\eqref{piin2} hold for $\frac{K}{\log\log x}\sqrt{\frac{x}{\log x}}\leq T$ when $T\geq T_0$. Here, $x_{\text{max}}$ is the largest value of $x$ for which this inequality holds when $T=T_0$.}
\begin{tabular}{|c|c|c|}
\hline
$T_0$ & $K$ & $x_{\text{max}}$ \\
\hline
$ 10^{13} $& $8.94$ & $1.335\cdot 10^{27}$\\
\hline
$ 10^{14} $& $8.76$ & $1.550\cdot 10^{29}$\\
\hline
$ 10^{15} $& $8.64$ & $1.762\cdot 10^{31}$\\
\hline
\end{tabular}
\label{ttable}
\end{table}

\subsection{Weakening the constant}

Using the methods in Section \ref{sectmain}, we can obtain weaker bounds that hold for larger ranges of $x$. Here, the main idea is to alter the definition of $\mathcal{E}_3(x)$ and thereby change the leading term in \eqref{e3eq}. Doing this with the constant changed from $1/8\pi$ to a selection of larger values, we obtained the following result.
\begin{theorem}\label{weakthm}
     Let $T>0$ be such that the Riemann hypothesis holds for zeros $\rho$ with $0<\Im(\rho)\leq T$. Then, for corresponding values of $a$ and $K$ in Table \ref{constanttable}, the following estimates hold:
    \begin{align}
        |\psi(x)-x|&<a\sqrt{x}\log^2x,&\text{for }x\geq 3,\label{psiin2}\\
        |\theta(x)-x|&<a\sqrt{x}\log^2x,&\text{for }x\geq 3,\label{thetain2}\\
        |\Pi(x)-\li(x)|&<a\sqrt{x}\log x,&\text{for }x\geq 2\label{piin12},\\
        |\pi(x)-\li(x)|&<a\sqrt{x}\log x,&\text{for }x\geq 2\label{piin22},
    \end{align}
    provided $K\sqrt{\frac{x}{\log^3x}}\leq T$.
\end{theorem}
\begin{proof}
    Let $(a,K)=(1,1.19)$. For other values of $a$ and $K$ the method of proof is essentially identical. We use the same general reasoning as in the proof of Theorem \ref{mainthm}. Hence we only describe the small modifications required in this setting.
    
    Firstly, the minimum values for $x$ appearing in \eqref{psiin2}--\eqref{piin22} were obtained by checking each expression manually up to the minimum values appearing in \eqref{psiin}--\eqref{piin2}. We then let 
    \begin{equation}\label{ceeq2}
        c(x)=\frac{1}{2}\log x+D,\qquad \varepsilon(x)=\frac{\log^{5/2}x}{E\sqrt{x}},
    \end{equation}
     initially setting $D=0$ and $E=2.4$. Each of the error terms $\mathcal{E}_1(x),\ldots,\mathcal{E}_5(x)$ changed slightly due to the new choice of $\varepsilon(x)$ in \eqref{ceeq2}. The main difference occurred with $\mathcal{E}_3(x)$ and $\mathcal{E}_4(x)$, now given by
    \begin{align*}
        \mathcal{E}_3(x)&:=\frac{\sqrt{x}}{2\pi}\left(\frac{1}{2}\log x+\log(\alpha)-2\log\log x\right)^2-a\sqrt{x}\log^2x,\\
        \mathcal{E}_4(x)&:=\beta\sqrt{x}\log^2x
    \end{align*}
    for some computable constants $\alpha$ and $\beta$. In particular, this definition of $\mathcal{E}_3(x)$ gives
    \begin{equation*}
        \sum_{0<|\Im(\rho)|\leq\frac{\sqrt{2c}}{\varepsilon}}\left|a_{c,\varepsilon}(\rho)\frac{x^{\rho}}{\rho}\right|\leq a\sqrt{x}\log^2x+\mathcal{E}_3(x).
    \end{equation*}
    For the iterative process we started with $A=1.101\cdot 10^{26}$ (as per Corollary \ref{maincor}), $B=1.2$, $C=2.017$, $D=0$ and $E=2.4$. Here, $B$ was such that the inequalities \eqref{psiin2}--\eqref{piin22} held for $B\sqrt{\frac{x}{\log(x)^3}}\leq T$ and $C$ was such that 
    \begin{align*}
        |\psi(x)-x|<a\sqrt{x}\log x(\log x-C),
    \end{align*}
    held for each $x$ in this range. For the second iteration we used
    \begin{equation*}
            (A,B,C,D,E)=(2.128\cdot 10^{30},1.19,2.015,0,2.38)
    \end{equation*}
    which gave the desired result.
\end{proof}
\begin{remark}
    In the above proof we fixed $D=0$. A small improvement is possible if we allowed $D$ to be negative. However, this requires reworking several inequalities from the proof of Theorem \ref{mainthm} so we decided not to do so here.
\end{remark}

\def\arraystretch{1.5}
\begin{table}[h]
\centering
\caption{Corresponding values of $a$ and $K$ for Theorem \ref{weakthm}. The value $x_{\text{max}}$ is the largest $x$ for which the inequalities \eqref{psiin2}--\eqref{piin22} hold upon setting $T=3\cdot 10^{12}$.}
\begin{tabular}{|c|c|c||c|c|c|}
\hline
$a$ & $K$ & $x_{\text{max}}$ & $a$ & $K$ & $x_{max}$\\
\hline
$ 1 $& $1.19$ & $2.165\cdot 10^{30}$ & $ 10^{4} $& $1.16\cdot 10^{-4}$ & $4.723\cdot 10^{38}$\\
\hline
$ 10 $& $0.117$ & $2.738\cdot 10^{32}$ & $10^5$ & $1.16\cdot 10^{-5}$ & $5.522\cdot 10^{40}$\\
\hline
$ 100 $& $0.0116$ & $3.360\cdot 10^{34}$ & $10^6$ & $1.16\cdot 10^{-6}$ & $6.404\cdot 10^{42}$\\
\hline
$ 1000 $& $0.00116$ & $4.004\cdot 10^{36}$ & $10^7$ & $1.16\cdot 10^{-7}$ &$7.375\cdot 10^{44}$\\
\hline
\end{tabular}
\label{constanttable}
\end{table}

\section{An inequality of Ramanujan}\label{appsec}
In one of his notebooks, Ramanujan proved that the inequality
\begin{equation}\label{rameq}
    \pi(x)^2<\frac{ex}{\log x}\pi\left(\frac{x}{e}\right)
\end{equation}
holds for sufficiently large $x$ (see \cite[pp 112--114]{berndt2012ramanujan}). Several authors (\cite{dudek2015solving}, \cite{axler2017estimates}, \cite{platt2021error}, \cite{hassani2021remarks}) have attempted to make \eqref{rameq} completely explicit. It is widely believed that the last integer counterexample occurs at $x=38,358,837,682$. In fact, this follows under assumption of the Riemann hypothesis \cite[Theorem 1.3]{dudek2015solving}. 

The best unconditional result is due to Platt and Trudgian \cite[Theorem 2]{platt2021error}. In particular, they show that \eqref{rameq} holds for both $38,358,837,683\leq x\leq\exp(58)$ and $x\geq\exp(3915)$. Our bounds on $\pi(x)$ allow for a significant improvement on the first of these results. To demonstrate this, we use a simple (but computationally intensive) method to verify \eqref{rameq}, obtaining the following result.
\begin{theorem}\label{ramthm}
    For $38,358,837,683\leq x\leq\exp(103)$, Ramanujan's inequality \eqref{rameq} holds unconditionally.
\end{theorem}
\begin{proof}
    For $38,358,837,682<x\leq\exp(43)$, the theorem follows from \cite[Theorem 3]{axler2017estimates}. Platt and Trudgian also prove \eqref{rameq} for $\exp(43)<x\leq\exp(58)$ but the author thought it would be instructive to re-establish their result.
    
    So, let $x>\exp(43)$ and write $z=\log x$. Then \eqref{rameq} is equivalent to
    \begin{equation}\label{zeq}
        \frac{e^{z+1}}{z}\pi(e^{z-1})-\pi(e^z)^2>0.
    \end{equation}
    Set $a=1/8\pi$. By Theorem \ref{mainthm} we have that $|\pi(x)-\li(x)|<a\sqrt{x}\log x$ for $\exp(43)<x\leq\exp(59)$. Thus, \eqref{zeq} is true in this range provided
    \begin{equation}\label{lieq}
        \frac{e^{z+1}}{z}\li(e^{z-1})-\frac{a(z-1)}{ z}e^{\frac{3z+1}{2}}-\left(\li(e^z)+aze^{z/2}\right)^2>0
    \end{equation}
    for $43<z\leq 59$. We write
    \begin{equation*}
        f(z)=\frac{e^{z+1}}{z}\li(e^{z-1}),\quad g(z)=\frac{a(z-1)}{ z}e^{\frac{3z+1}{2}}+\left(\li(e^z)+aze^{z/2}\right)^2
    \end{equation*}
    so that \eqref{lieq} is equivalent to $f(z)-g(z)>0$. Note that $f(z)$ and $g(z)$ are both increasing for $z>1$. Hence, if $f(z_0)>g(z_0+\delta)$ for some $z_0>1$ and $\delta>0$, then $f(z)>g(z)$ for every $z\in(z_0,z_0+\delta)$. We thus performed a ``brute force" verification by setting $\delta=5\cdot 10^{-8}$ and showing that
    \begin{equation*}
        f(43)-g(43+\delta)>0,\ f(43+\delta)-g(43+2\delta)>0,\ \ldots,\ f(59-\delta)-g(59)>0.
    \end{equation*}
    This was achieved using a short algorithm written in Python. The computations took just under a day on a 2.4GHz laptop.
    
    We then repeated the above argument using Theorem 4.1 with $a=1$ and a smaller $\delta=2.5\cdot 10^{-8}$. This proved \eqref{rameq} for $\exp(59)<x\leq\exp(69)$. Continuing in this fashion for each value of $a$ in Table 2 we see that \eqref{rameq} holds in the range $\exp(43)<x\leq\exp(103)$ as desired.
\end{proof}
Certainly one could extend Table \ref{constanttable} and the computations in the above proof. However this would require a large amount of computation time. Thus, to improve on Theorem \ref{ramthm} the author suggests switching to a more sophisticated and less computational method. For instance, one could attempt to modify the arguments in \cite[Section 6]{axler2017estimates} or \cite[Section 5]{platt2021error}.

\section{Future work}
There are several ways in which one could expand on the work in this paper, for instance:
\begin{enumerate}[label=(\arabic*)]
    \item One could produce a wider range of weakened bounds similar to those in Theorem \ref{weakthm}. For example, one could provide a more general expression for $K$ as a function of $a$.
    \item One could produce analogous results for primes in arithmetic progressions. To do this, one would need to rework the results in this paper and \cite{buthe2016estimating} using computations of zeros of Dirichlet $L$-functions (e.g.\ \cite{platt2016numerical}) and the explicit formula for $\psi(x,\chi)$ \cite[Chapter 19]{davenport2013multiplicative}. Then, if desired, one could also consider other types of $L$-functions.
    \item As discussed in Section \ref{appsec}, it is possible to improve Theorem \ref{ramthm} with some work. It would be interesting to optimise the results of this paper and those in \cite{platt2021error} to see how close one could get to making Ramanujan's inequality completely explicit.
\end{enumerate}

\section*{Acknowledgements}
Thanks to my supervisor Tim Trudgian for all of his wonderful suggestions and insights on this project.

\printbibliography
\end{document}